\documentclass{amsart}
\usepackage{mathrsfs,latexsym,amsfonts,amssymb}
\usepackage{hyperref}

\newtheorem{theorem}{Theorem}[section]
\newtheorem{lemma}[theorem]{Lemma}
\newtheorem{corollary}[theorem]{Corollary}

\theoremstyle{definition}
\newtheorem{definition}[theorem]{Definition}

\theoremstyle{remark}

\newtheorem{example}[theorem]{Example}

\begin{document}
\title[Uniform bases at non-isolated points and maps]%
{Uniform bases at non-isolated points and maps}

\author{Fucai Lin}
\address{Fucai Lin: Department of Mathematics,
Zhangzhou Normal University, Zhangzhou 363000, P. R. China; Department of Mathematics,
Sichuan University, Chengdou, 610064, P. R. China}
\email{linfucai2008@yahoo.com.cn; lfc19791001@163.com}
\author{Shou Lin}
\address{Shou Lin: Department of Mathematics,
Zhangzhou Normal University, Zhangzhou 363000, P. R. China;
Institute of Mathematics, Ningde Teachers' College, Ningde, Fujian
352100, P. R. China}
\email{linshou@public.ndptt.fj.cn}

\thanks{Supported in part by the NSFC(No. 10571151).}

\keywords{Perfect mappings; uniform bases at non-isolated points;
open mappings; developable at non-isolated points.}%insert keywords
\subjclass[2000]{54C10; 54D70; 54E30; 54E40}%insert subject class

%\date{\today}
\begin{abstract}
In this paper, the authors mainly discuss the images of spaces with
an uniform base at non-isolated points, and obtain the following
main results: (1)\ Perfect maps preserve spaces with an uniform base
at non-isolated points; (2)\ Open and closed maps preserve regular
spaces with an uniform base at non-isolated points; (3)\ Spaces with
an uniform base at non-isolated points don't satisfy the
decomposition theorem.
\end{abstract}

\maketitle
\section{Introduction}
Recently, spaces with an uniform base or spaces with a sharp base
bring some topologist attention and interesting results about
certain bases are obtained \cite{AJRS, BB, MH}. In \cite{LL}, the
authors define the notion of uniform bases at non-isolated points
and obtain some related matters. For example, it is proved that a
space $X$ has an uniform base at non-isolated points if and only if
$X$ is the open boundary-compact image of a metric space. It is well
known that the class of spaces under the open and compact images of
metric spaces are preserved by perfect maps or closed and open
maps(see \cite{MH}). Hence a question arises:``What kind of maps
preserve spaces with a uniform base at non-isolated points?'' In
this paper we shall consider the invariance of spaces with an
uniform base at non-isolated points under perfect maps or closed and
open maps.

By $\mathbb{R, N}$, denote the set of all real numbers and positive
integers, respectively. For a topological space $X$, let $\tau (X)$
denote the topology for $X$, and let
$$I(X)=\{x:x \mbox{ is an isolated point of } X\},$$
$$X^{d}=X-I(X),$$
$$\mathcal{I}(X)=\{\{x\}:x\in I(X)\},$$
$$\mathcal{I}_\Delta(X)=\{(\{x\}, \{x\}):x\in I(X)\}.$$

In this paper all spaces are Hausdorff, all maps are continuous and
onto. Recall some basic definitions.

\begin{definition}
Let $\mathcal{P}$ be a base of a space $X$. $\mathcal{P}$ is an {\it
uniform base} \cite{Al}~(resp. {\it uniform base at non-isolated
points} \cite{LL}) for $X$ if for each ($resp.$\ non-isolated) point
$x\in X$ and $\mathcal{P}^{\prime}$ is a countably infinite subset
of $\{P\in\mathcal{P}:x\in P\}$, $\mathcal{P}^{\prime}$ is a
neighborhood base at $x$ in $X$.
\end{definition}

In the definition, ``at non-isolated points'' means ``at each
non-isolated point of $X$''.

\begin{definition}\cite{En}
Let $f:X\rightarrow Y$ be a map.
\begin{enumerate}
\item $f$ is a {\it boundary-compact map}, if each
$\partial f^{-1}(y)$ is compact in $X$;

\item $f$ is a {\it compact map} if each $f^{-1}(y)$ is compact in $X$;

\item $f$ is a {\it perfect map} if $f$ is a closed and compact map.
\end{enumerate}
\end{definition}

\begin{definition}
Let $X$ be a space and $\{\mathcal{P}_{n}\}_{n}$ a sequence of
collections of open subsets of $X$.
\begin{enumerate}
\item $\{\mathcal{P}_{n}\}_{n}$ is called
a {\it quasi-development}~\cite{Be} for $X$ if for every $x\in U$
with $U$ open in $X$, there exists $n\in \mathbb{N}$ such that
$x\in\mbox{st}(x,\mathcal{P}_{n})\subset U$.

\item $\{\mathcal{P}_{n}\}_{n}$ is called a {\it
development}~\cite{WJ}(resp. {\it development at non-isolated
points}\cite{LL}) for $X$ if $\{\mbox{st}(x,\mathcal{P}_{n})\}_{n}$
is a neighborhood base at $x$ in $X$ for each~(resp. non-isolated)
point $x\in X$.

\item $X$ is called {\it quasi-developable}~(resp. {\it developable},
{\it developable at non-isolated points}) if $X$ has a
quasi-development (resp. development, development at non-isolated
points).
\end{enumerate}
\end{definition}

Obviously, in the definition about developments at non-isolated
points we can assume that each $\mathcal{P}_{n}$ is a cover for $X$.
Also, it is easy to see that a space which is developable at
non-isolated points is quasi-developable, but a space with a
development at non-isolated points may not have a development, see
Example in \cite{LL}.

\begin{definition}
Let $\mathcal P$ be a family of subsets of a space $X$.\
$\mathcal{P}$ is called {\it point-finite at non-isolated points}
\cite{LL} if for each non-isolated point $x\in X$, $x$ belongs to at
most finite elements of $\mathcal{P}$.\ Let $\{\mathcal P_n\}_{n}$
be a development (resp. a development at non-isolated points) for
$X$. $\{\mathcal P_n\}_{n}$ is said to be a {\it point-finite
development} (resp. {\it a point-finite development at non-isolated
points}) for $X$ if each $\mathcal P_n$ is point-finite at each
(resp. non-isolated) point of $X$.
\end{definition}

Readers may refer to \cite{En, Ls} for unstated definitions and
terminology.

\maketitle
\section{Developments at non-isolated points}
In this section some characterizations of spaces with a development
at non-isolated points are established.

Let $X$ be a topological space. $g:\mathbb{N}\times X\rightarrow
\tau(X)$ is called a $g$-function, if $x\in g(n, x)$ and $g(n+1,
x)\subset g(n, x)$ for any $x\in X$ and $n\in \mathbb{N}$. For
$A\subset X$, put $$g(n,A)=\bigcup_{x\in A}g(n, x).$$

\begin{theorem}
Let $X$ be a topological space. Then the following conditions are
equivalent:
\begin{enumerate}
\item $X$ has a development at non-isolated points;

\item There exists a $g$-function for $X$ such that, for every $x\in X^{d}$ and
sequences $\{x_{n}\}_{n}, \{y_{n}\}_{n}$ of $X$, if $\{x,
x_{n}\}\subset g(n, y_{n})$ for every $n\in\mathbb{N}$, then
$x_{n}\rightarrow x.$

\item $X$ is a quasi-developable space, and $X^{d}$ is a perfect subspace of
$X$.
\end{enumerate}
\end{theorem}

\begin{proof}
$(1)\Rightarrow (2).$\ \ Let $\{\mathcal{U}_{n}\}_{n}$ be a
development at non-isolated points for $X$. We can assume that
$\mathcal{I}(X)\subset \mathcal{U}_{n}$ for every $n\in \mathbb{N}$.

For every $x\in X$ and $n\in \mathbb{N}$, fix $U_{n}\in
\mathcal{U}_{n}$ with $x\in U_{n}$, where $U_{n}=\{x\}$ when $x\in
I(X)$. Let $g(n, x)=\bigcap_{i<n}U_{i}$. Then $g:\mathbb{N}\times
X\to\tau(X)$ is a $g$-function for $X$. For every $x\in X^{d}$, if
sequences $\{x_{n}\}_{n}$, $\{y_{n}\}_{n}$ satisfy $\{x,
x_{n}\}\subset g(n, y_{n})$ for every $n\in\mathbb{N}$, then
$x_{n}\rightarrow x$ because $\{\mathcal{U}_{n}\}_{n}$ is a
development at non-isolated points.

$(2)\Rightarrow (3).$\ \ Let $g$ be a $g$-function with (2). Put
$\mathcal{U}_{n}=\{g(n, x):x\in X^{d}\}$ for every $n\in
\mathbb{N}$. Then $\{\mathcal{U}_{n}\}_{n}\cup\{\mathcal{I}(X)\}$ is
a quasi-development for $X$. Otherwise, there exist $x\in X^{d}$ and
an open neighborhood $U$ of $x$ in $X$ such that $\mbox{st}(x,
\mathcal{U}_{n})\not\subset U$ for every $n\in \mathbb{N}$. For
every $n\in\mathbb{N}$, choose $x_{n}\in\mbox{st}(x,
\mathcal{U}_{n})-U$, then there exists $y_{n}\in X$ such that
$\{x_{n}, x\}\subset g(n, y_{n})$. Thus $x_{n}\rightarrow x$, a
contradiction as $X-U$ is closed. Hence $X$ has a quasi-development.

For any closed subset $B$ of $X^{d}$, it is obvious that $B\subset
\bigcap_{n\in \mathbb{N}}(g(n, B)\cap X^{d})$. If a point $x\in
\bigcap_{n\in \mathbb{N}}(g(n, B)\cap X^{d})-B$, then $x\in g(n,
B)\cap X^{d}$ for every $n\in \mathbb{N}$. There exists a sequence
$\{y_{n}\}_{n}$ in $B$ such that $\{x, y_{n}\}\subset g(n, y_{n})$,
so $y_{n}\rightarrow x$ by (2). Since $X^{d}$ is closed in $X$, $B$
is closed in $X$, then $x\in B$, a contradiction. Thus
$B=\bigcap_{n\in \mathbb{N}}(g(n, B)\cap X^{d})$, and $X^{d}$ is a
perfect subspace for $X$.

$(3)\Rightarrow (1).$\ \ Let $\{\mathcal{U}_{n}\}_{n}$ be a
quasi-development for $X$, and $X^{d}$ be a perfect subspace of $X$.
For any $n\in \mathbb{N}$, there exists a sequence $\{F_{n,
j}\}_{j}$ of closed subsets of $X^{d}$ such that $(\cup
\mathcal{U}_{n})\cap X^{d}=\bigcup_{j\in\mathbb{N}}F_{n, j}$. For
each $n, j\in\mathbb{N}$, put $$\mathcal{H}_{n,
j}=\mathcal{U}_{n}\cup\{X-F_{n, j}\}.$$ Then $\{\mathcal{H}_{n,
j}\}_{n, j}$ is a development at non-isolated points for $X$.
Indeed, for any $x\in X^{d}$ and $x\in U\in\tau$, since
$\{\mathcal{U}_{n}\}_{n}$ is a quasi-development for $X$, there
exists $n\in \mathbb{N}$ such that $x\in \mbox{st}(x,
\mathcal{U}_{n})\subset U$. Hence there exists $j\in\mathbb{N}$ such
that $x\in F_{n, j}$. Thus $x\in \mbox{st}(x, \mathcal{H}_{n,
j})\subset U$ because $x\notin X-F_{n, j}$.
\end{proof}

Let $\mathcal{P}$ be a pair-family of subsets of $X$. For any $P\in
\mathcal{P}$, we denote $P=(P^{\prime}, P^{\prime\prime})$. For any
$\mathcal{R}\subset \mathcal{P}$, denote
$$\mathcal{R}^{\prime}=\{P^{\prime}:P\in\mathcal{R}\},$$
$$\mathcal{R}^{\prime\prime}=\{P^{\prime\prime}:P\in\mathcal{R}\},$$
$$\mbox{st}(x, \mathcal{R})=\cup\{P^{\prime\prime}:P\in\mathcal{R},
x\in P^{\prime}\},\ \ x\in X,$$
$$\mbox{st}(A,
\mathcal{R})=\cup\{P^{\prime\prime}:P\in\mathcal{R}, A\cap
P^{\prime}\neq\emptyset\},\ \ A\subset X.$$ For each $i\leq n$ and
$\mathcal{R}_{i}\subset \mathcal{P}$, denote
$$\mathcal{R}_{1}\wedge\mathcal{R}_{2}\cdots\wedge\mathcal{R}_{n}=
\{(\bigcap_{i\leq n}P_{i}^{\prime}, \bigcap_{i\leq
n}P_{i}^{\prime\prime}):P_{i}\in \mathcal{R}_{i}, i\leq n\}.$$

\begin{definition}\cite{Bu}
Let $X$ be a topological space and $\mathcal{P}$ a pair-family for
$X$. $\mathcal{P}$ is called a {\it pair-network} if $\mathcal{P}$
satisfies the following conditions:

(i)\ \ $P^{\prime}\subset P^{\prime\prime}$ for any $(P^{\prime},
P^{\prime\prime})\in \mathcal{P}$;

(ii)\ \ For any $x\in U\in \tau (X)$, there exists $(P^{\prime},
P^{\prime\prime})\in \mathcal{P}$ such that $x\in P^{\prime}\subset
P^{\prime\prime}\subset U$.
\end{definition}

\begin{theorem}
Let $X$ be a space. Then the following conditions are equivalent:
\begin{enumerate}

\item $X$ is a developable space at non-isolated points;

\item There exists a pair-network $\bigcup_{n\in\mathbb{N}}\mathcal{P}_{n}$ for $X$ satisfying the following conditions:

\noindent(i)\ For every $n\in\mathbb{N}$,
$\mathcal{P}_{n}^{\prime}|_{X^{d}}$ is a closed and locally finite
family in $X^{d}$, and $\mathcal{P}_{n}^{\prime\prime}$ is open in
$X$;

\noindent(ii)\ For every compact subset $K$ and $K\subset U\in\tau
(X)$, there exists $m\in\mathbb{N}$ such that $K\subset\mbox{st}(K,
\mathcal{P}_{m})\subset U$.

\item There exists a pair-network $\bigcup_{n\in\mathbb{N}}\mathcal{P}_{n}$ for $X$ satisfying the following conditions:

\noindent(i)\ For every $n\in\mathbb{N}$,
$\mathcal{P}_{n}^{\prime}|_{X^{d}}$ is a closed and locally finite
family in $X^{d}$;

\noindent(ii)\ For every $x\in U\in \tau(X)$, there exists
$m\in\mathbb{N}$ such that $x\in\mbox{st}^{\circ}(x,
\mathcal{P}_{m})\subset U$.
\end{enumerate}
\end{theorem}

\begin{proof}
We only need to prove that $(3)\Rightarrow (1)\Rightarrow (2)$.

$(3)\Rightarrow (1)$.\ \ Let $X$ had a pair-network $\bigcup_{n\in\mathbb{N}}\mathcal{R}_{n}$ with (3). Then
\begin{small} $\bigcup_{n\in \mathbb{N}}\mathcal{R}_{n}^{\prime}|_{X^{d}}$ \end{small}is a closed and $\sigma$-locally finite
network in $X^{d}$, $X^{d}$ is a perfect subspace of $X$.

For any $n, k\in\mathbb{N}$, let

$\phi_{n, k}=\{\mathcal{F}\subset
\mathcal{R}_{n}^{\prime}|_{X^{d}}:|\mathcal{F}|=k\}$;

$U(\mathcal{F})=(\cup\{R^{\prime\prime}:R\in \mathcal{R}_{n},
R^{\prime}\cap X^{d}\in
\mathcal{F}\})^{\circ}-\cup(\mathcal{R}_{n}^{\prime}|_{X^{d}}-\mathcal{F})$,
where $\mathcal{F}\in \phi_{n, k}$;

$\mathcal{U}_{n, k}=\{U(\mathcal{F}):\mathcal{F}\in\phi_{n, k}\}.$

\noindent We should prove that $\{\mathcal{U}_{n, k}\}_{n, k}\cup
\{\mathcal{I}(X)\}$ is a quasi-development for $X$. For any $x\in
X^{d}$ and $x\in U\in \tau (X)$, there exists $m\in \mathbb{N}$ such
that $x\in \mbox{st}^{\circ}(x, \mathcal{R}_{m})\subset U$. Let
$$\mathcal{F}=\{R^{\prime}\cap X^{d}:R\in \mathcal{R}_{m}, x\in R^{\prime}\}, |\mathcal{F}|=k.$$
It is easy to see $\mathcal{F}\in \phi_{m, k}$. Hence $x\in
U(\mathcal{F})\subset \mbox{st}^{\circ}(x, \mathcal{R}_{m})\subset
U$. If $\mathcal{G}\in \phi_{m, k}-\{\mathcal{F}\}$, then $x\in\cup
(\mathcal{R}_{m}^{\prime}|_{X^{d}}- \mathcal{G})$. Thus $x\notin
U(\mathcal{G})$. So $x\in U(\mathcal{F})=\mbox{st}(x,
\mathcal{U}_{m, k})\subset U$. Hence $\{\mathcal{U}_{n, k}\}_{n,
k}\cup \{\mathcal{I}(X)\}$ is a quasi-development for $X$.

In a word,  $X$ has a development at non-isolated points by Theorem
2.1.

$(1)\Rightarrow (2)$.\ \ Let $\{\mathcal{U}_{n}\}_{n}$ be a
development at non-isolated points for $X$. We can also assume that
$\{\mathcal{U}_{n}\}_{n}$ satisfies the following conditions (a)-(c)
for every $n\in\mathbb{N}$:

(a)\ $\mathcal{U}_{n+1}$ refines $\mathcal{U}_{n}$;

(b)\ $\mathcal{I}(X)\subset \mathcal{U}_{n}$;

(c)\ $U_{1}\cap X^{d}\neq U_{2}\cap X^{d}$ for any distinct $U_{1},
U_{2}\in \mathcal{U}_{n}-\mathcal{I}(X)$.

\noindent Put
$\mathcal{U}_{n}-\mathcal{I}(X)=\{U_{\alpha}:\alpha\in\Lambda_{n}\}$.
Since $X^{d}$ is a developable subspace of $X$, it is a
subparacompact subspace, then there exists a collection
$\mathcal{F}_{n}=\bigcup_{k\in\mathbb{N}}\mathcal{F}_{n,k}$ of
subsets of $X^{d}$ such that each $\mathcal{F}_{n,k}=\{F_{k,
\alpha}:\alpha\in\Lambda_{n}\}$ is closed and discrete in $X^{d}$
and $F_{k, \alpha}\subset U_{\alpha}\cap X^{d}$ for every $k\in
\mathbb{N}, \alpha\in\Lambda$. Let $$\mathcal{P}_{n, k}=\{(F_{k,
\alpha},
U_{\alpha}):\alpha\in\Lambda_{n}\}\cup\mathcal{I}_{\Delta}(X).$$
Then $\bigcup_{n, k\in\mathbb{N}}\mathcal{P}_{n, k}$ is a
pair-network for $X$. Let $$\mathcal{H}(k_{1}, k_{2},\cdots
,k_{n})=\bigwedge_{i\leq n}\mathcal{P}_{i, k_{i}},\ \
k_{i}\in\mathbb{N}, i\leq k.$$ Then $\mathcal{H}(k_{1}, k_{2},\cdots
,k_{n})$ satisfies the condition (i) in (2). Suppose that $K\subset
U$ with $K$ compact and $U$ open in $X$. If $x\in K\cap X^{d}$,
there exists a sequence $\{k_{i}\}_{i}$ in $\mathbb{N}$ such that
$x\in \cup\mathcal{F}_{i, k_{i}}$ for any $i\in \mathbb{N}$. For
every $n\in \mathbb{N}$, put
$$A_{n}=\cup\{H^{\prime}:H\in \mathcal{H}(k_{1}, k_{2},\cdots
,k_{n}), H^{\prime}\cap K\neq\emptyset, H^{\prime\prime}\not\subset
U\}.$$ Since $X^{d}$ is closed in $X$, $\{A_{n}\}_{n}$ is a
decreasing sequence of closed subsets of $X$. Then there exists
$m\in \mathbb{N}$ such that $A_{m}=\emptyset$. Otherwise, there
exist a non-isolated point $y\in K\cap (\bigcap_{n\in
\mathbb{N}}A_{n})$ and $j\in \mathbb{N}$ such that $\mbox{st}(y,
\mathcal{U}_{j})\subset U$. Thus $$\mbox{st}(y, \mathcal{H}(k_{1},
k_{2},\cdots ,k_{j}))\subset \mbox{st}(y, \mathcal{U}_{j})\subset
U.$$ This is a contradiction with the definition of $A_{j}$. Hence
$A_{m}=\emptyset$ for some $m\in \mathbb{N}$, and $$x\in\mbox{st}(K,
\mathcal{H}(k_{1}, k_{2},\cdots ,k_{m}))\subset U.$$ By the
compactness of $K$,
$\cup\{\mathcal{H}(k_{1},\cdots,k_{n}):n,k_{i}\in \mathbb{N}, i\leq
n\}$ satisfies the condition (ii) of (2).
\end{proof}

\begin{corollary}
$X$ is a developable space at non-isolated points if and only if $X$
has a pair-network
$\mathcal{P}=\bigcup_{n\in\mathbb{N}}\mathcal{P}_{n}$ satisfying the
following conditions:

(i)\ For any $n\in\mathbb{N}$, $\mathcal{I}_{\Delta}(X)\subset
\mathcal{P}_{n}$, and $P^{\prime}\subset X^{d}$ for any $P\in
\mathcal{P}_{n}-\mathcal{I}_{\Delta}(X)$;

(ii)\ For every $n\in\mathbb{N}$,
$\mathcal{P}_{n}^{\prime}|_{X^{d}}$ is a closed and hereditarily
closure-preserving family in $X^{d}$;

(iii)\ There exists $m\in\mathbb{N}$ such that
$x\in\mbox{st}^{\circ}(x, \mathcal{P}_{m})\subset U$ for any $x\in
U\in \tau (X)$.
\end{corollary}

\begin{proof}
Necessity. It is easy to see by the proof of $(1)\Rightarrow (2)$ in
Theorem 2.3.

Sufficiency. Let
$\mathcal{P}=\bigcup_{n\in\mathbb{N}}\mathcal{P}_{n}$ be a
pair-network for $X$ satisfying the condition (i)-(iii). For any
$n\in \mathbb{N}$, put
$$D_{n}=\{x\in X:|(\mathcal{P}_{n}^{\prime})_{x}|\geq\aleph_{0}\},$$
$$\mathcal{R}_{n}=\{(\overline{P^{\prime}-D_{n}}, P^{\prime\prime}):
P\in \mathcal{P}_{n}-\mathcal{I}_{\Delta}(X)\}$$
$$\hspace{1.5cm}\cup \{(\{x\}, \mbox{st}(x, \mathcal{P}_{n})):x\in
D_{n}\}\cup\mathcal{I}_{\Delta}(X).$$ Then
$\bigcup_{n\in\mathbb{N}}\mathcal{R}_{n}$ is a pair-network for $X$.
We shall show that $\bigcup_{n\in\mathbb{N}}\mathcal{R}_{n}$
satisfies the condition (3) in Theorem 2.3. Since $X$ is a
first-countable space by (iii), it is easy to see that
$\mathcal{R}_{n}^{\prime}|_{X^{d}}$ is a closed and locally finite
family in $X^{d}$ by \cite[Lemma 3.2.16]{Ls}. Suppose $x\in
U\in\tau(X)$. If $x\in I(X)\cup (\bigcup_{n\in\mathbb{N}}D_{n})$, it
is obvious that there exists $m\in \mathbb{N}$ such that
$x\in\mbox{st}^{\circ}(x, \mathcal{R}_{m})\subset U$. If $x\in
X-(I(X)\cup (\bigcup_{n\in\mathbb{N}}D_{n}))$, then
$x\in\mbox{st}(x, \mathcal{R}_{n})=\mbox{st}(x, \mathcal{P}_{n})$.
Thus $X$ is a developable space at non-isolated points by Theorem
2.3.
\end{proof}

\begin{example}
Let $X=\mathbb{N}\cup\{p\}$, here $p\in\beta\mathbb{N}-\mathbb{N}$,
endowed with the subspace topology of Stone-\v{C}ech
compactification $\beta\mathbb{N}$. Then $X^{d}=\{p\}$ is a
metrizable subspace of $X$. Since $X$ is not first-countable, then
$X$ does not have a development at non-isolated points.
\end{example}

\maketitle
\section{The images of spaces with an uniform base at non-isolated points}

In this section invariant properties of spaces with a development at
non-isolated points and spaces with an uniform base at non-isolated
points are discussed under perfect maps or closed and open maps.

A space $X$ is called {\it metacompact} if every open cover of $X$
has a point-finite open refinement.

\begin{lemma}
For a space $X$, $X^{d}$ is a metacompact subspace of $X$ if and
only if every open cover of $X$ has an open refinement which is
point-finite at non-isolated points.
\end{lemma}

\begin{proof}
Sufficiency is obvious. We only prove the necessity.

Necessity. Let $X^{d}$ be a metacompact subspace of $X$. For every
open cover $\mathcal{U}$ for $X$, it is easy to see that
$\mathcal{U}|_{X^{d}}$ is an open cover for subspace $X^{d}$. Since
$X^{d}$ is a metacompact subspace, there exists an open and
point-finite refinement $\mathcal{V}$(in $X^{d}$) for
$\mathcal{U}|_{X^{d}}$. For every $V\in \mathcal{V}$, there exist
$U\in \mathcal{U}$ and $W(V)\in\tau (X)$ such that $V=W(V)\cap
X^{d}$ and $W(V)\subset U$. Put
$$\mathcal{W}=\{W(V):V\in \mathcal{V}\}.$$
Then $\mathcal{W}$ is an open refinement for $\mathcal{U}$ and also
point-finite at non-isolated points.
\end{proof}

\begin{lemma}
Let $X$ be a topological space. Then the following conditions are
equivalent:
\begin{enumerate}
\item $X$ is an open boundary-compact image of a metric space;

\item $X$ has an uniform base at non-isolated points;

\item $X$ has a point-finite development at non-isolated points;

\item $X$ has a development
at non-isolated points, and $X^{d}$ is a metacompact subspace of
$X$.
\end{enumerate}
\end{lemma}

\begin{proof}
$(1)\Leftrightarrow (2)\Leftrightarrow (3)$ was proved in
\cite{LL}.\ We only need to prove $(1)\Rightarrow (4)\Rightarrow
(3)$.

$(1)\Rightarrow (4)$.\ \ Let $f:M\rightarrow X$ be an open
boundary-compact mapping, where $M$ is a metric space. Let
$\mathcal{U}$ be an open cover for $X$. Then $f^{-1}(\mathcal{U})$
is an open cover for $M$. Since $M$ is paracompact, there exists a
locally finite open refinement $\mathcal{V}$ of
$f^{-1}(\mathcal{U})$. It is easy to see that $f(\mathcal{V})$ is
point-finite at non-isolated points, and refines $\mathcal{U}$.
Hence $X^{d}$ is metacompact by Lemma 3.1.

$(4)\Rightarrow (3)$.\ \ Let $\{\mathcal{U}_{n}\}_{n}$ be a
development at non-isolated points of $X$. For every $n\in
\mathbb{N}$, since $X^{d}$ is metacompact, $\mathcal{U}_{n}$ has an
open refinement $\mathcal{V}_{n}$ which is point-finite at
non-isolated points. Hence $\{\mathcal{V}_{n}\}_{n}$ is a
point-finite development at non-isolated points.
\end{proof}

Let $\bigcup_{n\in \mathbb{N}}\mathcal{P}_{n}$ be a pair-network for
a space $X$. We say that $\bigcup_{n\in \mathbb{N}}\mathcal{P}_{n}$
satisfies ($\star$) if it has the (i) of Corollary 2.4. That is, let
($\star$) be the condition:

($\star$) For any $n\in\mathbb{N}$, $\mathcal{I}_{\Delta}(X)\subset
\mathcal{P}_{n}$ and $P^{\prime}\subset X^{d}$ for any $P\in
\mathcal{P}_{n}-\mathcal{I}_{\Delta}(X)$.

\begin{theorem}
Spaces with a development at non-isolated points are preserved by
perfect maps.
\end{theorem}

\begin{proof}
Let $f:X\rightarrow Y$ be a perfect map, where $X$ is developable at
non-isolated points. Let $\bigcup_{n\in \mathbb{N}}\mathcal{P}_{n}$
be a pair-network which satisfies the condition (2) in Theorem 2.3
for $X$. It is easy to see that we can suppose that
$\bigcup_{n\in\mathbb{N}}\mathcal{P}_{n}$ satisfies the condition
($\star$) by the proof of $(1)\Rightarrow (2)$ in Theorem 2.3.

For any $n\in \mathbb{N}$, put
$$\mathcal{B}_{n}=\{(f(P^{\prime}), f(P^{\prime\prime})):P\in
\mathcal{P}_{n}\};$$
$$\mathcal{R}_{n}=\{(f(P^{\prime})\cap Y^{d},
f(P^{\prime\prime})):P\in
\mathcal{P}_{n}-\mathcal{I}_{\Delta}(X)\}\cup\mathcal{I}_{\Delta}(Y).$$
Since $f$ is closed, $Y^{d}\subset f(X^{d})$. It is easy to check
that $\bigcup_{n\in \mathbb{N}}\mathcal{R}_{n}$ is a pair-network
for $Y$. Next, we shall show that it satisfies the condition (3) of
Theorem 2.3 for $Y$.

(i)\ It is well-known that a locally finite family is preserved by a
perfect map. Since $f|_{X^{d}}:X^{d}\rightarrow f(X^{d})$ is a
perfect map and $\mathcal{P}'_{n}|_{X^{d}}$ is closed and locally
finite in $X^{d}$, $\{f(P'\cap X^{d}): P\in\mathcal{P}_{n}\}$ is
closed and locally finite in $f(X^{d})$, then $$\{f(P'\cap
X^{d})\cap Y^{d}:
P\in\mathcal{P}_{n}-\mathcal{I}_{\Delta}(X)\}=\mathcal{R}_{n}^{\prime}|_{Y^{d}}$$
is closed and locally finite in $Y^{d}$ by the condition ($\star$).

(ii)\ Let $y\in U\in\tau (Y)$. We can suppose that $y\in Y^{d}$.
Since $f^{-1}(y)$ is compact for $X$, there exists $m\in \mathbb{N}$
such that
$$f^{-1}(y)\subset \mbox{st}(f^{-1}(y),
\mathcal{P}_{m})\subset f^{-1}(U).$$ Since $f$ is closed and
$\mbox{st}(f^{-1}(y), \mathcal{P}_{m})$ is open in $X$, then
$$y\in\mbox{st}^{\circ}(y, \mathcal{B}_{m})\subset\mbox{st}(y, \mathcal{B}_{m})\subset U.$$
If $y\in f(P')\cap Y^{d}$ with $P\in\mathcal{I}_{\Delta}(X)$,
$f(P'')=\{y\}\subset\mbox{st}(y, \mathcal{R}_{m})$. Thus
$\mbox{st}(y, \mathcal{B}_{m})= \mbox{st}(y, \mathcal{R}_{m})$,
hence $y\in\mbox{st}^{\circ}(y, \mathcal{R}_{m})\subset U$.
\end{proof}

\begin{corollary}
Spaces with an uniform base at non-isolated points are preserved by
perfect maps.
\end{corollary}

\begin{proof}
Since metacompactness is preserved by closed maps, it is easy to see
by Lemma 3.2 and Theorem 3.3.
\end{proof}

Let $\Xi$ be a topological property. $\Xi$ is said to satisfy {\it
the decomposition theorem} if, for any space $X$ with the property
$\Xi$ and any closed map $f:X\rightarrow Y$, there exists a
$\sigma$-closed discrete subset $Z\subset Y$ such that $f^{-1}(y)$
is compact in $X$ for any $y\in Y- Z$.

In \cite[Theorem 1.1]{Ch2}, J. Chaber proved that each regular
$\sigma$-space satisfies the decomposition theorem.

\begin{theorem}
Let $f:X\rightarrow Y$ be a closed map, where $X$ is a regular space
having a development at non-isolated points. If $Y$ is a
first-countable space, then $Y$ is developable at non-isolated
points.
\end{theorem}

\begin{proof}
Since subspace $X^{d}$ is a Moore space, there exists a subspace
$Z=\bigcup_{n\in\mathbb{N}}Z_{n}\subset Y^{d}$ such that, for any
$y\in Y^{d}- Z$, $f^{-1}(y)\cap X^{d}$ is a compact subset of
$X^{d}$ by \cite[Theorem 1.1]{Ch2}, where each $Z_{n}$ is closed and
discrete in $Y^{d}$. Hence $f^{-1}(y)\cap X^{d}$ is a compact subset
of $X$ for any $y\in Y^{d}- Z$. For any $y\in Z$, let $\{U(y,
n):n\in\mathbb{N}\}$ be a neighborhood base of $y$ in $Y$. Let
$\bigcup_{n\in\mathbb{N}}\mathcal{P}_{n}$ be a pair-network for $X$
satisfying the condition (2) of Theorem 2.3, and the condition
($\star$) by the proof of $(1)\Rightarrow (2)$ in Theorem 2.3.

For any $n, j\in \mathbb{N}$, let
$$\mathcal{W}_{n}=\{(f(P^{\prime}), f(P^{\prime\prime})):P\in
\mathcal{P}_{n}\},$$
$$\mathcal{R}_{n}=\{(f(P^{\prime})\cap Y^{d},
f(P^{\prime\prime})):P\in
\mathcal{P}_{n}-\mathcal{I}_{\Delta}(X)\}\cup
\mathcal{I}_{\Delta}(Y),$$
$$\mathcal{H}_{n,j}=\{(\{y\}, U(y,
j)):y\in Z_{n}\}\cup \mathcal{I}_{\Delta}(Y).$$ Then
$$(\bigcup_{n\in\mathbb{N}}\mathcal{R}_{n})\cup(\bigcup_{n,
j\in\mathbb{N}}\mathcal{H}_{n,j})\cup\mathcal{I}_{\Delta}(Y)$$ is a
pair-network for $Y$ and satisfies the conditions (i) and (ii) of
Corollary 2.4 because a hereditarily closure-preserving family is
preserved by a closed map. We only need to prove that it also
satisfies (iii) in Corollary 2.4. For any $y\in U\in \tau (Y)$, we
discuss the following three cases respectively.

(a)\ \ If $y\in Z$, then there exist $n\in \mathbb{N}$ and
$j\in\mathbb{N}$ such that $y\in Z_{n}$ and $U(y, j)\subset U$.
Hence $y\in\mbox{st}^{\circ}(y, \mathcal{H}_{n,j})\subset U(y,
j)\subset U$.

(b)\ \ If $y\in Y^{d}- Z$, then $f^{-1}(y)\cap X^{d}$ is a compact
subset for $X$. There exists $m\in \mathbb{N}$ such that
$$f^{-1}(y)\cap X^{d}\subset\mbox{st}(f^{-1}(y)\cap X^{d},
\mathcal{P}_{m})\subset f^{-1}(U),$$ then
$$f^{-1}(y)\subset\mbox{st}(f^{-1}(y),
\mathcal{P}_{m})$$ $$=\mbox{st}(f^{-1}(y)\cap X^{d},
\mathcal{P}_{m})\cup\mbox{st}(f^{-1}(y)\cap I(X),
\mathcal{P}_{m})\subset f^{-1}(U),$$ thus $y\in\mbox{st}^{\circ}(y,
\mathcal{W}_{m})\subset U$. Since $\mbox{st}(y, \mathcal{R}_{m})=
\mbox{st}(y, \mathcal{W}_{m})$, $y\in\mbox{st}^{\circ}(y,
\mathcal{R}_{m})\subset U$.

(c)\ \ If $y\in I(Y)$, then $y\in\mbox{st}(y,
\mathcal{I}_{\Delta}(Y))=\{y\}\subset U$.

Hence $Y$ is a developable space at non-isolated points by Corollary
2.4.
\end{proof}

\begin{corollary}
Regular spaces with an uniform base at non-isolated points are
preserved by open and closed maps.
\end{corollary}

\begin{proof}
Let $f:X\rightarrow Y$ be an open and closed map, where $X$ is a
regular space having an uniform base at non-isolated points. Since
$f$ is open and closed, $Y$ is regular and first-countable space,
thus $Y$ has an uniform base at non-isolated points by Theorem 3.5.
\end{proof}

A collection $\mathcal C$ of subsets of an infinite set $D$ is said
to be {\it almost disjoint} if $A\cap B$ is finite whenever
$A\not=B\in\mathcal C$. Let $\mathcal A$ be an almost disjoint
collection of countably infinite subsets of $D$ and maximal with
respect to the properties. Isbell-Mr\'{o}wka space $\psi(D)$ is the
set $\mathcal A\cup D$ endowed with a topology as follows \cite{Mr}:
The points of $D$ are isolated. Basic neighborhoods of a point
$A\in\mathcal A$ are the sets of the form $\{A\}\cup (A-F)$ where
$F$ is a finite subset of $D$.

\begin{example}
There exists a closed map $f:X\rightarrow Y$, where $X$ is a regular
space with an uniform base at non-isolated points and $Y$ is a
first-countable space. However, $f$ is not a boundary-compact map.
\end{example}

\begin{proof}
Let $\mathcal A$ be an almost disjoint collection of countably
infinite subsets of $\mathbb{N}$ and maximal with respect to the
properties. Let $\psi(\mathbb{N})=\mathcal{A}\cup\mathbb{N}$ be the
Isbell-Mr\'{o}wka space. Then $\psi(\mathbb{N})$ is a regular space
with an uniform base at non-isolated points.

Define $f:\psi(\mathbb{N})\rightarrow \psi(\mathbb{N})/\mathcal A$
by a quotient map, then $f$ is a closed map and the quotient space
$\psi(\mathbb{N})/\mathcal A$ is a first-countable space. Since
$\partial f^{-1}(\{\mathcal A\})=\mathcal A$ is discrete in
$\psi(\mathbb{N})$, $f$ is not boundary-compact.
\end{proof}

Since a regular space with an uniform base is a $\sigma$-space,
regular spaces with an uniform base satisfy the decomposition
theorem. But regular spaces with an uniform base at non-isolated
points don't satisfy the decomposition theorem.

\begin{example}
There are a regular space $X$ with an uniform base at non-isolated
points and a closed map $f:X\to Y$ such that $f$ does not satisfy
the decomposition theorem.

Let $Y$ be the Isbell-Mr\'{o}wka space $\psi (D)$, where $D$ is an
uncountable set. Let $S_{1}=\{0\}\cup\{1/n:n\in \mathbb{N}\}$ be the
subspace of the real line $\mathbb{R}$. Put $$X=Y\times S_{1}-
(D\times\{0\}),$$ endowed with the subspace topology of product
topology. Then $X$ is a regular space. Let $f:X\rightarrow Y$ be the
projective map. Then $f$ is a closed map.

Let $\psi(D)=\mathcal{A}\cup D$, where
$\mathcal{A}=\{A_{\alpha}\}_{\alpha\in \Lambda}$ and each
$A_{\alpha}=\{x(\alpha, n):n\in \mathbb{N}\}\subset D$. Put

\hspace{1.5cm}$V_{n}(\alpha)=\{x(\alpha, m):m\geq
n\}\cup\{A_{\alpha}\}$,

\hspace{1.5cm}$U_{n}(0)=\{0\}\cup\{1/m:m\geq n\}$,

\hspace{1.5cm}$\mathcal{B}=\{\{(x, y)\}:(x, y)\in D\times (S_{1}-
\{0\})\}$

\hspace{2.5cm}$\cup\{V_{n}(\alpha)\times U_{n}(0):n\in\mathbb{N},
\alpha\in \Lambda\}\cup \{V_{m}(\alpha)\times \{1/n\}:m,
n\in\mathbb{N}\}$.\\
It is easy to see that $\mathcal{B}$ is an uniform base at
non-isolated points for $X$. However, $f^{-1}(y)=\{y\}\times
(S_{1}-\{0\})$ is not compact in $X$ for any $y\in D$. Since any
closed (in $Y$) subset contained in $D$ is finite, $D$ is not a
$\sigma$-discrete subspace for $Y$. Thus $f:X\rightarrow Y$ does not
satisfy the decomposition theorem.
\end{example}

The authors would like to thank the referee for his/her valuable
suggestions.

\end{document}